\documentclass[preprint]{elsarticle}

\usepackage{graphicx}
\usepackage{amsmath}
\usepackage{amsthm}
\usepackage{amssymb}
\usepackage[numbers]{natbib}
\usepackage{bbm}
\usepackage{bm}
\usepackage{natbib}
\usepackage[svgnames]{xcolor}

\newtheorem{theorem}{Theorem}
\newtheorem{lemma}{Lemma}

\newtheorem{proposition}{Proposition}

\journal{}

\newtheorem{remark}{Remark}[section]

\begin{document}

\begin{frontmatter}
\title{Weak convergence of marked point processes\\ generated by crossings of 
multivariate jump processes. Applications to neural network modeling.} 

\author[tambo]{M.Tamborrino\corref{cor1}}
\ead{mt@math.ku.dk}
\address[tambo]{Department of Mathematical Sciences, Copenhagen University, Universitetsparken 5, Copenhagen, Denmark }

\author[sac1]{L.Sacerdote}
\ead{laura.sacerdote@unito.it}
\address[sac1]{Department of Mathematics \lq\lq G. Peano\rq\rq, University of Turin, V. Carlo Alberto 10, Turin, Italy}
 
\author[tambo]{M. Jacobsen}
\ead{martin@math.ku.dk}

\cortext[cor1]{Corresponding author. Phone: +45 35320785. Fax: +45 35320704.}

\begin{abstract}
We consider the multivariate point process determined by the crossing times of the components of a multivariate jump process through a multivariate boundary, assuming to reset each component to an initial value after its boundary crossing. We prove that this point process converges weakly to the point process determined by the crossing times of the limit process. This holds for both diffusion and deterministic limit processes. The almost sure convergence of the first passage times under the almost sure convergence of the processes is also proved. The particular case of a multivariate Stein process converging to a multivariate Ornstein-Uhlenbeck process is  discussed as a guideline for applying  diffusion limits for jump processes. We apply our theoretical findings to neural network modeling. The proposed model gives a mathematical foundation to the generalization of the class of Leaky Integrate-and-Fire models for single neural dynamics to the case of a firing network of  neurons. This will help future study of dependent spike trains.
\end{abstract}

\end{frontmatter}

\section{Introduction}\label{intro}
Limit theorems for weak convergence of probability measures and stochastic jump processes with frequent jumps of small amplitudes have been widely investigated in the literature, both for univariate and multivariate processes. Besides the pure mathematical interest, the main reason is that these theorems allow to switch from discontinuous to continuous processes, improving their mathematical tractability. Depending on the assumptions on the frequency and size of the jumps, the limit object can be either deterministic, obtained e.g. as solution of systems of ordinary/partial differential equations \cite{Kurtz1,Kurtz2,PTW,RTW}, or stochastic \citep{BilConv,Jacod,Skorohod1,Ricciardi}. Limit theorems of the first type are usually called the fluid limit, thermodynamic limit or hydrodynamic limit, and give rise to what is called Kurtz approximation \cite{Kurtz2}, see e.g. \cite{DarlingRev} for a review. In this paper we consider limit theorems of the second type, which we refer to as diffusion limits, since they yield diffusion processes. Some well known univariate examples are the Wiener, the Ornstein-Uhlenbeck (OU) and the Cox–-Ingersoll-–Ross (also known as square-root) processes, which can be obtained as diffusion limits of random walk \cite{BilConv}, Stein \cite{Lansky} and branching processes \cite{Feller}, respectively. A special case of weak convergence of multivariate jump processes is considered in Section \ref{Section2}, as a guideline for applying the method proposed in \cite{Jacod}, based on convergence of triplets of characteristics.

In several applications, e.g. engineering \citep{eng}, finance \citep{applbook,ZhangCo}, neuroscience \citep{Coombes,ReviewSac}, physics \citep{physFPT,bookFPT} and reliability theory \citep{rel,PieperDom}, the stochastic process  evolves in presence of a boundary, and it is of paramount interest to detect the so-called first-passage-time (FPT) of the process, i.e. the epoch when the process crosses a boundary level for the first time. A natural question arises: how does the FPT of the jump process relate to the FPT of its limit process? The answer is not trivial, since the FPT is not a continuous functional of the process and therefore the continuous mapping theorem cannot be applied. \\
There exist different techniques for proving the weak convergence of the FPTs of univariate processes, see e.g. \citep{Lansky,KallianpurFPT}. The extension of these results to multivariate processes requests to define the behavior of the single component after its FPT. Throughout, we assume to reset it and then restart its dynamics. This choice is suggested by application in neuroscience and reliability theory, see e.g. \cite{STZ,LLC}. The collection of FPTs coming from different components determine a multivariate point process, which we interpret as a univariate marked point process in Section 3. \newline
The primary aim of this paper is to show that the marked point process determined by the exit times of a multivariate jump process with reset converges weakly to the marked point process determined by the exit times of its limit process (cf. Section \ref{Section4}  and Section \ref{Section5} for proofs). Interestingly this result does not depend on whether the limit process is obtained through a diffusion or a Kurtz approximation. Moreover, we also prove that the almost sure convergence of the processes guarantees the almost sure convergence of their passage times.

The second aim of this paper is to provide a simple mathematical model to describe a neural network  able to reproduce dependences between spike trains, i.e. collections of a short-lasting events (spikes) in which the electrical membrane potential of a cell rapidly rises and falls. The availability of such a model can be useful in neuroscience as a tool for the study of the neural code. Indeed it is commonly believed that the neural code is encoded in the firing times of the neurons: dependences between spike trains correspond to the transmission of information from a neuron to others \cite{Perkel1, Seg}. Natural candidates as neural network models are generalization of univariate Leaky Integrate-and-Fire (LIF) models, which describe single neuron dynamics, see e.g. \cite{ReviewSac,Tuckwell1989}. These models sacrifice realism, e.g. they disregard the anatomy of the neuron, describing it as a single point, and the biophysical properties related with ion channels, for mathematical tractability  \cite{Burkitt1,Burkitt2,Gerstner}. Thought some criticisms have appeared \cite{68}, they are considered good descriptors of the neuron spiking activity \cite{67,71}. \\
In Section \ref{SectionX} we interpret our processes and theorems in the framework of neural network modeling, extending the class of LIF models from univariate to multivariate. First, the weak convergence shown in Section \ref{Section2} gives a neuronal foundation to the use of multivariate OU processes for modeling sub-threshold membrane potential dynamics of neural networks \cite{BBB,Brunel},   where dependences between neurons are determined by common synaptic inputs from the surrounding network. Second, the multivariate process with reset introduced in Section \ref{Section3} defines the firing mechanism for a neural network. Finally, the weak convergence of the univariate marked point process proved in Section \ref{Section4} guarantees that  the neural code is kept under the diffusion limit. The paper is concluded with a brief discussion and outlook on further developments and applications. 

\section{Weak convergence of multivariate Stein processes to multivariate Ornstein-Uhlnebeck}\label{Section2}
As an example for proving the weak convergence of multivariate jump processes using the method proposed in \cite{Jacod}, we show the convergence of a multivariate Stein to a multivariate OU. Mimicking the one-dimensional case \citep{Ricciardi,Lansky}  we introduce a sequence of multivariate Stein processes $\left (\bm X_n\right)_{n\geq 1}$, with $\bm{X}_{n}=\left\{ (X_{1;n},\ldots ,X_{k;n})(t); t\geq 0\right\}$ originated in the starting position $\bm x_{0;n}=(x_{01;n},\ldots, x_{0k;n})$. For each $1\leq j \leq k, n\geq 1$, the $j$th component of the multivariate Stein process, denoted by $X_{j;n}(t)$, is defined by 
\begin{eqnarray}
\nonumber X_{j;n}(t) &=&x_{0j;n}-\int_{0}^{t}\frac{X_{j;n}(s)}{\theta }\,ds+\left[
a_nN_{j;n}^{+}(t)+b_nN_{j;n}^{-}(t)\right] \\
\label{Stein} &+&\sum_{A\in \mathcal{A}}\mathbbm{1}_{\{j\in A\}}\left[
a_nM_{A;n}^{+}(t)+b_nM_{A;n}^{-}(t)\right] ,
\end{eqnarray}
where $\mathbbm{1}_A$ is the indicator function of the set $A$ and  $\mathcal{A}$ denotes the set of all subsets of $\left\{1,\ldots ,k\right\}$ consisting of at least two elements. Here  $N_{j;n}^{+}$ (intensity $\alpha _{j;n}$), $N_{j;n}^{-}$ (intensity $\beta _{j;n}$), $M_{A;n}^{+}$ (intensity $\lambda _{A;n}$) and $M_{A;n}^{-}$ (intensity $\omega _{A;n}$) for $1\leq j\leq k, A\in \mathcal{A}$, are a sequence of independent Poisson processes.   In particular, the processes $N^+_{j;n}(t)$
and $N^-_{j;n}(t)$ are typical of the $j$th component, while the processes  $M^+_{A;n}(t)$ and $M^+_{A;n}(t)$ act on a set of 
components $A \in \mathcal{A}$. Therefore, the dynamics of $X_{j;n}$ are determined by two different types of inputs. Moreover, $a_n>0$ and $b_n<0$ denote the constant amplitudes of the inputs $N_n^+, M_n^+$ and $N_n^-, M_n^-$, respectively.
\begin{remark}
The process defined by \eqref{Stein} is an example of  piecewise-deterministic Markov process  or stochastic hybrid system, i.e. a process with deterministic behavior between jumps \cite{Davis,Jacobsen}.
\end{remark}

 For each $A \in \mathcal{A}, 1 \leq j \leq k$ and  
\begin{eqnarray}  
\label{cond3}&& \alpha_{j; n}\to \infty, \qquad \beta_{j; n}\to \infty, \qquad \lambda_{A;n} \to \infty, \qquad \omega_{A;n} \to \infty, \\
\label{cond4}&& a_n \to 0,\qquad b_n \to 0,
\end{eqnarray}
we assume that the rates of the Poisson processes fulfill
\begin{eqnarray}  
\label{cond1}&&\mu_{j; n}=\alpha_{j; n} a_n + \beta_{j; n} b_n \to \mu_{j}, \qquad \mu_{A;n}=\lambda_{A;n}a_n + \omega_{A;n}b_n \to \mu_{A},\\ 
\label{cond2}&&\sigma_{j; n}^2 =
\alpha_{j; n} a_n^2 + \beta_{j; n} b_n^2 \to \sigma_j ^2,
\qquad \sigma_{A;n}^2 = \lambda_{A;n}a_n^2+\omega_{A;n}b_n^2 \to \sigma_{A}^2,
\end{eqnarray}
as $n\to \infty$. A possible parameter choice satisfying these conditions is
\begin{eqnarray*}
a_n&=& - b_n = \frac{1}{n}\\
\alpha_{j;n}&=&(\mu_j+\frac{\sigma^2_j}{2}n)n, \qquad \beta_{j;n}=\frac{\sigma^2_j}{2}n^2,\qquad 1\leq j\leq k\\
\lambda_{A;n}&=&(\mu_A+\frac{\sigma^2_A}{2}n)n, \qquad \omega_{A;n}=\frac{\sigma^2_A}{2}n^2, \qquad A\in \mathcal{A}.
\end{eqnarray*}
\begin{remark}
Jumps possess amplitudes decreasing to zero for $n\to \infty$ but occur at an increasing frequency roughly inversely proportional to the square of the jump size, following the literature for univariate diffusion limits. Thus we are not in the fluid limit setting, where the frequency are roughly inversely proportional to the jump size and the noise term is proportional to $1/\sqrt{n}$ \cite{Kurtz1}. 
\end{remark}
To prove the weak convergence of $\bm{X}_n$, we first define a new process $\bm{Z}_n=\left\{(Z_{1;n},\ldots,Z_{k;n})(t); t\geq 0\right\}$, with $j$th component given by
\begin{equation*}
Z_{j;n}(t)=-\Gamma_{j;n}t+\left[a_n N^+_{j;n}(t)+b_nN^-_{j;n}(t)\right]+\sum_{A\in \mathcal{A}}\mathbbm{1}_{\{j\in A\}} 
\left[a_nM^{+}_{A;n}(t)+b_nM^{-}_{A;n}(t)\right],
\end{equation*}
with 
\[
\Gamma _{j;n}=\mu _{j;n}+\sum_{A\in \mathcal{A}}\mathbbm{1}_{\{j\in A\}}\mu _{A;n},\qquad 1\leq j \leq k.
\]
The process $\bm Z_n$ converges weakly to a Wiener process $\bm{W}=\left\{(W_1,\ldots,W_n)(t); t\geq 0\right\}$: 
\begin{lemma}\label{lemma1}
Under conditions $(\ref{cond3}), (\ref{cond4}), (\ref{cond1}), (\ref{cond2})$, $\bm{Z}_n$ converges weakly to a  multivariate Wiener process $\bm{W}$ with mean $\bm{0}$ and  definite positive not diagonal covariance matrix $\bm\Psi$ with components
\begin{equation}\label{sigma} 
\psi_{jl}=\mathbbm{1}_{\{j=l\}}\sigma^2_j+\sum_{A\in \mathcal{A}}\mathbbm{1}_{\{j,l\in A\}}\sigma _{A}^{2}, \qquad 1\leq j,l \leq k.
\end{equation}
\end{lemma}
\noindent The proof of Lemma \ref{lemma1} is given in \ref{AppA}. Note that $Z_{j;n}(t)$ is the martingale part of $X_{j;n}(t)$, see \eqref{eq11}. Thus martingale limit theorems can alternatively be used for proving Lemma \ref{lemma1},  mimicking  the proofs in \cite{PTW,RTW}.\newline 
Finally, we show that $\bm{X}_n$ is a continuous functional of $\bm{Z}_n$, and it holds 
\begin{theorem}\label{theo}
Let $\bm x_{0;n}$ be a sequence in $\mathbb{R}^k$ converging to  $\bm y_0=(y_{01},\ldots, y_{0k})$. Then, the sequence of processes $\bm{X}_n$ defined by (\ref{Stein}) with rates fulfilling $(\ref{cond1}), (\ref{cond2})$, under conditions (\ref{cond3}), (\ref{cond4}), converges weakly to the multivariate OU diffusion process $\bm{Y}$ given by 
\begin{equation}\label{OU}
Y_{j}(t)=y_{0j}+ \int_{0}^t \left[-\frac{Y_{j}(s)}{\theta}+ 
\Gamma_{j} \right] ds + W_{j}(t), \qquad 1\leq j\leq k,
\end{equation}
where $\Gamma_{j}$ is defined by
\begin{eqnarray}\label{drift}
\Gamma _{j}=\mu _{j}+\sum_{A\in \mathcal{A}} \mathbbm{1}_{\{j \in A\}}\mu _{A}, \qquad 1\leq j \leq k,
\end{eqnarray}
and $\bm{W}$ is a $k$-dimensional Wiener process with mean $\bm{0}$ and covariance matrix $\bm \Psi$ given by \eqref{sigma}.
\end{theorem}
\noindent The proof of Theorem \ref{theo} is given in \ref{AppA}.
\begin{remark}
If all $\sigma_j^2$ and $\sigma_A^2$ in \eqref{cond2} equal 0, Theorem \ref{theo} yields a deterministic (fluid) limit and results from \cite{PTW} can be applied. 
\end{remark}
\begin{remark}
Theorem \ref{theo} also holds when $\left(\bm x_{0;n}\right)_{n\geq 1}$ is a random sequence converging to a random vector $\bm y_0$.
\end{remark}
\begin{remark}\label{Remark2.4}
The obtained OU process can be rewritten as 
\begin{equation}\label{OUeq}
d\bm{Y}(t)=(-\bm{C}\bm{Y}(t) + \bm{D})dt + d\bm{W}(t),
\end{equation}
where $\bm{C}$ is a diagonal $k\times k$ matrix, $\bm D$ is a $k$-dimensional vector and $\bm W$ is a multivariate Wiener process with definite positive non-diagonal covariance matrix $\bm\Psi$ representing correlated Gaussian noise. For simulation purposes, the diffusion part in \eqref{OUeq} should be rewritten through the Cholesky decomposition. A modification of the original Stein model can be obtained introducing direct interactions between the $i$th and $j$th components. The corresponding diffusion limit process verifies \eqref{OUeq} with $\bm {C}$ non-diagonal matrix.
\end{remark}

\section{The multivariate FPT problem: preliminaries}\label{Section3}
Consider a sequence $(\bm{X}_n)_{n\geq 1}$ of multivariate jump processes weakly converging to $\bm{Y}$. Let $\bm{B}=(B_1,\ldots, B_k)$ be a $k$-dimensional vector of boundary values, where $B_j$ is the boundary of the $j$th component of the process. We denote $T_{j;n}$ the crossing time of the $j$th component of the jump process  through the boundary $B_j$, with $B_j>x_{0j;n}$. That is
\[
T_{j;n}=T_{B_j}(X_{j;n})=\inf \{ t>0: X_{j;n}(t)>B_j\}.
\]
Moreover, we denote $\tau_{1;n}$ the minimum of the FPTs of the multivariate jump process $\bm X_n$, i.e.
\[
 \tau_{1;n}=\min\left(T_{1;n},\ldots,T_{k;n}\right),
\]
and $\eta_{1;n}\subset \{1,\ldots,k\}$ the discrete random variable specifying the set of jumping components at time $\tau_{1;n}$. \newline
We introduce the reset procedure as follows. Whenever a component $j$ attains its boundary, it is instantaneously reset to  $r_{0j}<B_j$, and then it restarts, while the other components pursue their evolution till the attainment of their boundary. This procedure determines the new process $\bm{X}^*_n$. We define it  by introducing a sequence $\left( \bm{X}_n^{\left( m\right) }\right) _{m\geq 1}$ of multivariate jump processes defined on successive time windows, i.e. $\bm X^{(m)}_n$ is defined on the $m$th time window, for $m=1,2,\ldots$. Conditionally on $(\bm{X}^{(1)}_n,\ldots, \bm{X}_n^{(m)})$, $\bm{X}_n^{(m+1)}$ obeys to the same stochastic differential equation as $\bm{X}_n$, with random starting position determined by $(\bm{X}^{(1)}_n,\ldots, \bm{X}_n^{(m)})$. In particular, the first time window contains the process $\bm{X}_n$ up to $\tau_{1;n}$, which we denote by $\bm{X}_n^{(1)}$. The second time window contains the  process $\bm{X}_n^{(2)}$ whose components are originated in $\bm{X}_n^{(1)}(\tau_{1;n})$, except for the crossing components $\eta_{1;n}$,  which are set to their reset values. This second window lasts until when one of the component attains its boundary at time $\tau_{2;n}$. Successive time windows are analogously introduced, defining the corresponding processes.\newline
Similarly, we define $T_j$ and $\tau_1$ for the process $\bm Y$, while $\eta_1\in \{1,\ldots, k\}$ is defined as the discrete random variable  specifying the jumping component at time $\tau_1$, since simultaneous jumps do not occur for $\bm Y$. We define the reset process $\bm{Y}^*$ by introducing a sequence $\left( \bm{Y}^{\left( m\right) }\right) _{m\geq 1}$ of multivariate diffusion processes. Set $\bm{Y}^{\left( 1\right) }\equiv \bm{Y}$. Conditionally on $\left( \bm{Y}^{\left( 1\right) },\ldots ,\bm{Y}^{\left( m\right)}\right) $, $\bm{Y}^{\left( m+1\right) }$ obeys to the same stochastic differential equation as $\bm{Y}$, with random starting position  determined by $\left( \bm{Y}^{\left( 1\right)
},\ldots ,\bm{Y}^{\left( m\right) }\right) $ and with the $k$-dimensional Brownian motion $\bm W$  independent of $\left( \bm{Y}^{\left( 1\right) },\ldots ,\bm{Y}^{\left( m\right) }\right) $, for $m\geq 1$. Below we shall briefly say that $\bm{X}_n^{\left( m+1\right) }$ (or $\bm{Y}^{\left( m+1\right) }$) is obtained by conditional independence and then specify the initial value $\bm x_{0;n}$ (or $\bm y_{0}$).\newline
Now we formalize the recursive definition of $\bm X^*_n$ and $\bm Y^*$ on consecutive time windows. A schematic illustration of the involved variables is given in Fig. \ref{FigStein}.

\begin{figure}[t]
\includegraphics[width=\textwidth]{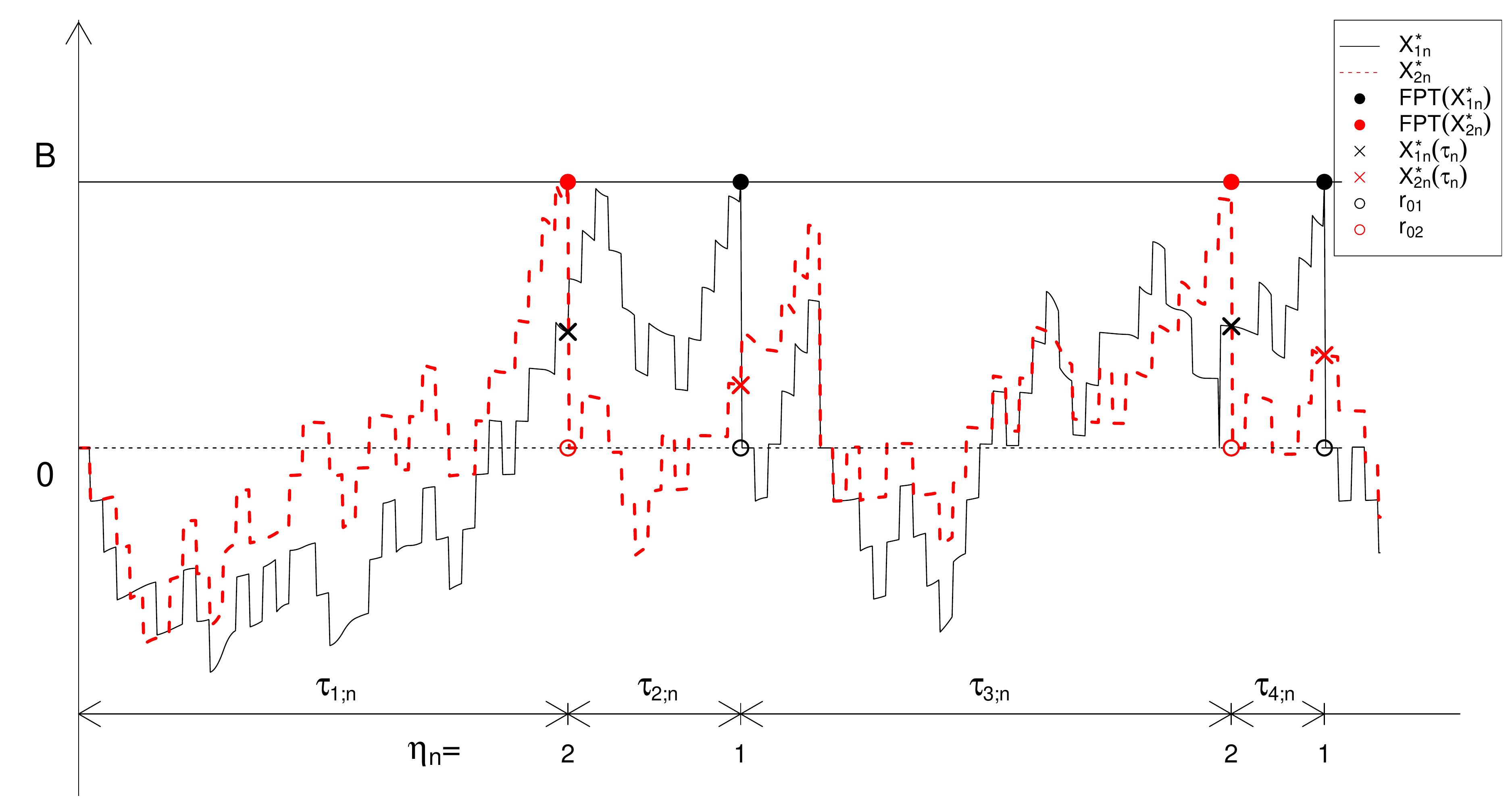}
\caption{Illustration of a bivariate jump process with reset $\bm{X}^*_n=(X_{1;n}^*,X_{2;n}^*)$. Whenever a component $j$ reaches its boundary $B_j$, it is instantaneously reset to its resting value $r_{0j}<B_j$. The process is defined in successive time windows determined by the FPTs of the process. Here $\eta_{i;n}$ denotes the set of jumping components at time $\tau_{i;n}$, which is the FPT of $\bm X^{(i)}_n$ in the $i$th time window.}
\label{FigStein}
\end{figure}

\begin{enumerate}
\item[Step $m=1$.] Define $\bm X_n^*(t) = \bm X_n(t)$ on the interval $[0,\tau_{1;n}[$ and $\bm Y^*(t)= \bm Y(t)$ on $[0,\tau_1[$, with resetting value $\bm X_n^*(0)=\bm r_0=\bm Y^*(0)$. Define $X_{j;n}^*(\tau_{1;n})=X_{j;n}(\tau_{1;n})$ if $j\not\in \eta_{1;n}$ or $X_{j;n}^*(\tau_{1;n})=r_{0j}$ if $j\in\eta_{1;n}$. Similarly define $Y_j^*(\tau_1)= Y_j(\tau_1)$ if $j\neq \eta_1$ or $Y_j^*(\tau_1)=r_{0j}$ if $j=\eta_1$.
\item[Step $m=2$.] For $j\in \eta _{1;n}$, obtain $\bm{X}_n^{\left( 2\right) }$ by conditional independence from $\bm{X}_n^{\left( 1\right) }$, with initial value $\bm{x}_{0;n}=\bm{X}_n^{*}\left( \tau
_{1;n}\right) $. Similarly, for $\eta _1=j$, obtain $\bm{Y}^{\left(
2\right) }$ by conditional independence from $\bm{Y}^{\left( 1\right) }$, with initial value $\bm{y}_0=\bm{Y}%
^{*}\left( \tau _1\right) $. Then, define $T_{j;n}^{ (2)}, \tau_{2;n}, \eta_{2;n}$ from $\bm{X}^{(2)}_n$ and $T_{j}^{(2)}, \tau_{2}, \eta_{2}$ from $\bm{Y}^{(2)}$, for $m=1$. Define $\bm{X}_n^*(t)=X^{(2)}_n(t-\tau_{1;n})$ on the interval $[\tau_{1;n},\tau_{1;n}+\tau_{2;n}[$ and $\bm Y^*(t)=\bm Y(t-\tau_1)$ on $[\tau_1,\tau_1+\tau_2[$. Then define $X_{j;n}^{*}\left( \tau _{1;n}+\tau _{2;n}\right)=X_{j;n}^{(2)}(\tau_{2;n})$ if $j\not\in \eta_{2;n}$ or $X_{j;n}^{*}\left( \tau _{1;n}+\tau _{2;n}\right)=r_{0j}$ if $j\in \eta_{2;n}$. Similarly define $Y_j^*\left(\tau_1+\tau_2\right)=Y_j^{(2)}(\tau_2)$ if $j\neq \eta_{2}$ or $Y_j^*\left(\tau_1+\tau_2\right)=r_{0j}$ if $j=\eta_2$.
\item[Step $m>2$.] For $j\in \eta_{m;n}$, obtain $\bm X^{(m)}_n$ by conditional independence from $\bm X^{(m-1)}$, with initial value $\bm x_{0;n}=\bm X^*_n(\sum_{l=1}^{m-1} \tau_{l;n})$. Similarly, for $\eta_m=j$, obtain $\bm Y^{(m)}$ by conditional independence from $\bm Y^{(m-1)}$, with initial value $\bm y_0=\bm Y^*(\sum_{l=1}^{m-1} \tau_{l})$. Define, $T_{j;n}^{(m)}, \tau_{m;n}, \eta_{m;n}$ from $\bm X^{(m)}_n$ and $T_{j}^{(m)}, \tau_{m}, \eta_{m}$ from $\bm Y^{(m)}$ as above. Define $\bm X_n^*(t)=\bm{X}_n^{(m)}(t-\sum_{l=1}^{m-1}\tau_{l;n})$ for $t \in [\sum_{l=1}^{m-1}\tau_{l;n}$, $\sum_{l=1}^{m}\tau_{l;n}[$ and $\bm Y^*(t)=\bm{Y}^{(m)}(t-\sum_{l=1}^{m-1}\tau_{l})$ for $t \in [\sum_{l=1}^{m-1}\tau_{l},\sum_{l=1}^{m}\tau_{l}[$. Then define $\bm{X}_{j;n}^{*}\left( \sum_{l=1}^m\tau_{l;n}\right)=X_{j;n}^{(r)}(\tau_{m;n})$ if $j\not\in \eta_{m;n}$ or $\bm{X}_{j;n}^{*}\left( \sum_{l=1}^m\tau_{l;n}\right)=r_{0j}$ if $j\in \eta_{m;n}$. Similarly define $\bm Y_j^*\left(\sum_{l=1}^m \tau_l\right)=Y_j^{(m)}(\tau_m)$ if $j\neq \eta_m$ or $\bm Y_j^*\left(\sum_{l=1}^m \tau_l\right)=r_{0j}$ if $j=\eta_m$.
\end{enumerate}
Besides the processes $\bm{X}^*_n$ and $\bm{Y}^*$, we introduce a couple of marked processes as follows. 
 Denote $\bm{\tau}_n=(\tau_{i;n})_{i\geq 1}, \bm{\tau}=(\tau_i)_{i\geq 1}, \bm{\eta}_n=(\eta_{i;n})_{i\geq 1}$ and $ \bm{\eta}=(\eta_i)_{i\geq 1}$. Then  $(\bm{\tau}_n,\bm{\eta}_n)$ and $(\bm{\tau},\bm{\eta})$ may be viewed as marked point processes describing the passage times of the processes $\bm X^*_n$ and $\bm Y^*$, respectively. These marked processes are superposition of point processes generated by crossing times of the single components.

\section{Main result on the convergence of the marked point process}\label{Section4}
The processes $\bm{X}_n^*$ and $\bm{Y}^*$ are neither continuous nor diffusions. Hence the convergence of $\bm{X}_n^*$ to $\bm{Y}^*$ does not directly follow from the convergence of $\bm{X}_n$ to $\bm{Y}$. Since the FPT is not  a continuous function of the process, the convergence of the marked point process $(\bm{\tau}_{n},\bm{\eta}_n)$ to $(\bm{\tau},\bm{\eta})$ has also to be proved. Proceed as follows. Consider the space $\mathcal{D}^{k}=\mathcal{D}([0,\infty[,\mathbb{R}^k)$, i.e. the space of functions $f:[0,\infty )\rightarrow \mathbb{R}^{k}$ that are right continuous and have a left limit at each $t\geq 0$, and the space $\mathcal{C}^{1}=\mathcal{C}\left( \left[ 0,\infty \right[ ,\mathbb{R}\right) $. For $y^{\circ }\in \mathcal{C}^{1}$, define the hitting time 
\[
\widetilde{T}_{B}\left( y^{\circ }\right) =\inf \left\{ t>0:y^{\circ
}(t)=B\right\},
\]%
and introduce the sets
\begin{eqnarray*}
H&=&\left\{ y^{\circ }\in \mathcal{C}^{1}:T_{B}\left( y^{\circ }\right) =%
\widetilde{T}_{B}\left( y^{\circ }\right) \right\},\\
H^k&=&\left\{ \bm y^\circ \in \mathcal{C}^k: T_{B_j}\left( y^{\circ }_j\right) =%
\widetilde{T}_{B_j}\left( y^{\circ }_j\right) \textrm {for all } 1\leq j \leq k\right\} .
\end{eqnarray*}
The hitting time $\widetilde T_B$ defines the first time when a process reaches $B$, while the FPT $T_B$ is defined as the first time when a process crosses $B$. Denote by\\ \lq\lq$\to$ in $\mathcal{D}^k$\rq\rq \ the  convergence of a sequence of functions in $\mathcal{D}^k$ and  by \lq\lq $\to$\rq\rq \ the ordinary convergence of a sequence of real numbers. To prove the main theorem, we need the following lemmas, whose proof are given in Section \ref{Section5}.
\begin{lemma}\label{FPTas}
Let $x_{n}^{\circ }$ belong to $\mathcal{D}^{1}$ for $n\geq 1$, and $%
y^{\circ }\in H$ with $y^{\circ }(0)<B$. If $x_{n}^{\circ }\rightarrow
y^{\circ }$ in  $\mathcal{D}^{1}$, then $T_{B}\left( x_{n}^{\circ }\right) \rightarrow T_{B}\left(
y^{\circ }\right)$.
\end{lemma}
\begin{lemma}\label{window1FPT}
Let $\bm x^\circ_n$ belong to $\mathcal{D}^k$ for $n\geq 1$,  $\bm y^\circ\in H^k$ with $\bm y^\circ(0)<\bm B$. If $\bm x^\circ_n\to \bm y^\circ$ in  $\mathcal{D}^{k}$, then  
\begin{equation}\label{triplets}
(\tau_{1;n}^\circ, \bm x^\circ_{n}(\tau^\circ_{1;n}),\eta^\circ_{1;n})\to (\tau^\circ_{1},\bm{y}^\circ(\tau^\circ_{1}),\eta^\circ_{1}).
\end{equation}
\end{lemma}
The weak convergence of the multivariate process with reset and of its marked point process corresponds to the weak convergence of the finite dimensional distributions of $(\bm{\tau}_n, \bm{X}^*_n(\bm{\tau}_n),\bm{\eta}_n)$ to
$(\bm{\tau}, \bm{Y}^*(\bm\tau),\bm{\eta})$, where $\bm \tau_n=(\tau_{i;n})_{i=1}^l$, $ \bm{X}^*_n(\bm\tau_n)=\left(\bm{X}^*_n(\tau_{i;n})\right)_{i=1}^l, \bm \eta_n =(\eta_{i;n})_{i=1}^l, \bm{\tau}=(\tau_i)_{i=1}^l,  \bm{Y}^*(\bm{t\tau}) =\left(\bm{Y}^*(\tau_{i})\right)_{i=1}^l$ and $\bm{\eta}=(\eta_i)_{i=1}^l$, for any $l \in \mathbb{N}$. We have
\begin{theorem}[Main theorem]\label{conspike}
The finite dimensional distributions of \\ $(\bm{\tau}_n, \bm{X}^*_n(\bm\tau_n),\bm{\eta}_n)$  converge weakly to those of $(\bm{\tau}, \bm{Y}^*(\bm\tau),\bm{\eta})$.
\end{theorem}
The proof of Theorem \ref{conspike} (cf. Section \ref{Section5}) uses the Skorohod's representation theorem \cite{Skorohod1} to switch the weak convergence of processes to almost sure convergence (strong convergence) in any time window between two consecutive passage times, which makes it possible to exploit Lemmas \ref{FPTas} and \ref{window1FPT}. As a consequence, the strong convergence of the processes implies the strong convergence of their FPTs.
\begin{remark}\label{remark2}
Theorem \ref{conspike} holds for any multivariate jump process weakly converging to a continuous process characterized by simultaneous hitting and crossing times for each component, i.e. $\tilde T_{B_j}=T_{B_j}$. Examples are diffusion processes and continuous processes with positive derivative at the epoch of the hitting time. 
\end{remark}
\begin{remark}\label{Remark4.1}
Both the weak convergence of $\bm{X}^*_n$ and of its marked point process also hold when the reset of the crossing component $j$ is not instantaneous, but happens with a delay $\Delta_j>0$, for $1\leq j\leq k$. This can be proved mimicking the proof of Theorem \ref{conspike}. 
\end{remark}

\section{Proof of the main results}\label{Section5}
\begin{proof}[Proof of Lemma \ref{FPTas}]
For each $s<T_{B}\left( y^{\circ }\right) $, $\sup_{t\leq s}y^{\circ }(t)<B$
and since $x_{n}^{\circ }\rightarrow y^{\circ }$ uniformly on $\left[ 0,s%
\right]$, also $\sup_{t\leq s}x_n^{\circ }(t)<B$ for $n$ sufficiently large.
This implies 
\begin{eqnarray*}
&&\lim \inf_{n\rightarrow \infty }T_{B}\left( x_{n}^{\circ }\right)  \geq s%
\text{ for all }s<T_{B}\left( y^{\circ }\right) \\
& \Rightarrow&\lim \inf_{n\rightarrow \infty }T_{B}\left( x_{n}^{\circ
}\right)  \geq T_{B}\left( y^{\circ }\right).
\end{eqnarray*}%
Because $y^{\circ }\in H$ we can find a sequence $t_k$ such that $t_{k}\downarrow T_{B}\left( y^{\circ
}\right) =\widetilde{T}_{B}\left( y^{\circ }\right) $ (with  $%
y^{\circ }\left( \widetilde{T}_{B}\left( y^{\circ }\right) \right) =B$) and $y^{\circ }\left( t_{k}\right) >B$ for all $k$. Since $x_{n}^{\circ
}\left( t_{k}\right) \rightarrow y^{\circ }\left( t_{k}\right) $ for all $k$%
, it follows that $T_{B}\left( x_{n}^{\circ }\right) \leq t_{k}$ for $n$
sufficiently large and therefore%
\begin{eqnarray*}
&\forall k,&\ \lim \sup_{n\rightarrow \infty }T_{B}\left( x_{n}^{\circ
}\right)  \leq t_{k} \\
&\Rightarrow& \lim \sup_{n\rightarrow \infty }T_{B}\left( x_{n}^{\circ
}\right)  \leq T_{B}\left( y^{\circ }\right) .
\end{eqnarray*}
\end{proof}

\begin{proof}[Proof of Lemma \ref{window1FPT}]
If $\eta^\circ_1=j$, then $\eta^\circ_{1;n}=j$ for $n$ large enough, since marginally $x_{i; n}^\circ\to  y^\circ_{i; n}$ for each component $1\leq i\leq k$. 
By Lemma \ref{FPTas} and since $y^\circ_{j}(0)<B_{j}$ by assumption, it follows 
\begin{equation}\label{astau}
\tau^\circ_{1;n}=T_{B_{j}}(x^\circ_{j;n})\to T_{B_j}(y^\circ_{j})=\tau^\circ_{1}.
\end{equation}
Moreover, it holds 
\begin{equation}\label{eqdis}
\vert x^\circ_{i;n}(\tau^\circ_{1;n})-y^\circ_i(\tau^\circ_{1})\vert \leq\vert x_{i;n}^\circ(\tau^\circ_{1;n}) - y^\circ_{i}(\tau^\circ_{1;n})\vert + \vert y^\circ_i(\tau^\circ_{1;n})-y^\circ_i(\tau^\circ_{1})\vert,
\end{equation}
which goes to zero when $n\to \infty$, for any $1\leq i\leq k$. Indeed, for each $s<\tau^\circ_1$, the convergence of $x^\circ_{i;n}$ to $y^\circ_i$ on a compact time interval $[0,s]$ implies the uniform convergence of $x^\circ_{i;n}$ to $y^\circ_i$ on $[0,s]$. Thus $x^\circ_{i;n}(\tau^\circ_{1;n})\to y^\circ_{i}(\tau^\circ_{1;n})$. From \eqref{astau} and since $y^\circ_i$ is continuous, $y^\circ_i(\tau^\circ_{1;n})\to y_i^\circ(\tau_1^\circ)$ when $n\to \infty$ for the continuous mapping theorem. Using the product topology on $\mathcal{D}^{k}$, we have that $\bm x^\circ_n\to \bm y^\circ$ in $\mathcal{D}^k$ if $x^\circ_{j; n}\to y^\circ_{j}$ in $\mathcal{D}^1$, for each $1\leq j\leq k$ \cite{Whittbook}, implying the lemma.
\end{proof}

Denote $E\stackrel{d}{=}F$ two random variables that are identically distributed. Then
\begin{proposition}\label{cor}
If a multivariate jump process $\bm X_n$ converges weakly to $\bm Y$, then there exist a probability space $(\Omega, \mathcal{F},\bm{P})$ and random elements  $\left(\widetilde{\bm  X}_n\right)_{n=1}^{\infty}$ and $\widetilde{\bm{Y}}$ in the Polish space $\mathcal{D}^k$, defined on  $(\Omega, \mathcal{F},\bm{P})$ such that $\bm{X}_n \stackrel{d}{=} \widetilde{\bm{X}}_n, \bm{X}\stackrel{d}{=}\widetilde{\bm{Y}}$ and $\widetilde{\bm{X}}_n \to \widetilde{\bm{Y}} \textit{ a.s. as } n\to \infty$.
\end{proposition}
\begin{proof}
From its definition, $\bm{X}_n$ belongs to $\mathcal{D}^{k}$, which is a Polish space with the Skorohod topology \citep{Lindvall}. Then, the proposition follows  applying the Skorohod's representation theorem \cite{Skorohod1}.
\end{proof}
\begin{proof}[Proof of Theorem \ref{conspike} (main result)]
Applying Theorem \ref{theo} and Proposition \ref{cor} in any time window between two consecutive passage times, there exist $\widetilde{\bm X}^*_n$ and $\widetilde{\bm Y}^*$ such that $\widetilde{\bm X}^*_n\overset{d}{=}\bm X^*_n, \widetilde{\bm Y}^*\overset{d}{=}\bm Y^*$ and $\widetilde{\bm X}^*_n \to \widetilde {\bm Y}^*$ a.s.. 
Define $\widetilde{\eta}_{j;n},\widetilde \tau_{j;n}$ from $\widetilde{\bm X}^*_n$ and 
$\widetilde{\eta}_{j},\widetilde \tau_{j}$ from $\widetilde{\bm Y}^*$ as done in Section \ref{Section3}. Assume $\widetilde\eta_m=j$ and thus $\widetilde\eta_{m;n}=j$ for $n$ sufficiently large, due to the strong convergence of the processes. If 
\begin{equation}\label{eqboh}
(\widetilde{\bm{\tau}}_n, \widetilde{\bm{X}}^*_n(\widetilde{\bm\tau}_n),\widetilde{\bm\eta}_n)\to
(\widetilde{\bm \tau}, \widetilde{\bm{Y}}^*(\widetilde{\bm\tau}),	\widetilde{\bm\eta}) \textrm { a.s. }
\end{equation}
holds, we would have
\begin{equation*}
 \widetilde{\tau}_{m;n}=T_{B_{j}}\left(\widetilde{X}^*_{j;n}\right) \overset{d}{=}T_{B_{j}}\left( X^*_{j;n}\right) =\tau_{m;n}, \quad  \widetilde{\tau}_m=T_{B_{j}}\left(\widetilde{Y}_{j}^*\right) \overset{d}{=}T_{B_{j}}\left( Y_{j}^*\right) =\tau_m,
\end{equation*}
since 
$\widetilde{\bm{X}}^*_{n}\overset{d}{=}\bm{X}^*_{n}$ and $\widetilde{\bm{Y}}^*\overset{d}{=} \bm{Y}^*$, which would also imply $\widetilde{\bm{X}}^*_n(\widetilde{\tau}_{m;n}) \overset{d}{=}\bm{X}^*_n(\tau_{m;n}) $ and $\widetilde{\bm{Y}}^*(\widetilde{\tau}_{m}) \overset{d}{=}\bm{Y}^*(\tau_{m}) $, for any $1\leq m\leq l$ and $l\in \mathbb{N}$, and thus the theorem. To prove \eqref{eqboh}, we proceed recursively in each time window:
\begin{enumerate}
\item[Step $m=1$.]  By definition, $\widetilde{\bm Y}^*$  behaves like a multivariate diffusion $\bm Y$ in $[0,\widetilde \tau_1[$. Since each one-dimensional diffusion component $\widetilde{Y}_j$ crosses the level $B_j$ infinitely often immediately after $\widetilde T_B(\widetilde Y_j)$, it follows $T_{B_j}(\widetilde Y_j)=\widetilde T_{B_j}(\widetilde Y_j)$,  for $1\leq j \leq k$ and thus $\widetilde{\bm Y^*} \in H^k$. Since also $\widetilde{\bm Y}^*(0)<\bm B$ by assumption, we can apply Lemma \ref{window1FPT} and obtain the convergence of the triplets \eqref{triplets} with not-reset firing components. This convergence also holds if we reset the firing components: assume $\widetilde\eta_1=j$ and then $\widetilde\eta_{1;n}=j$ for $n$ large enough. Then 
\begin{equation}\label{firingc}
\widetilde X^*_{j;n}(\widetilde\tau_{1;n})=r_{0;\widetilde\eta_{1;n}}=\widetilde Y^*_{j}(\widetilde\tau_1), 
\end{equation}
and thus $\widetilde{\bm X}^*_n (\widetilde\tau_{1;n}) \to\widetilde{\bm Y}^*(\widetilde\tau_{1}) $, implying \eqref{eqboh}.
\item[Step $m=2$.] On $[\widetilde\tau_{1;n},\widetilde\tau_{1;n}+\widetilde\tau_{2;n}[$, $\widetilde{\bm X}_n^*$ is obtained by conditionally independence from
$\widetilde{\bm X}_n^*$ on $[0,\widetilde\tau_{1;n}[$, with initial value $\widetilde{\bm x}_{0;n}=\widetilde{\bm X}^* (\tau_{1;n})$. Similarly, on $[\widetilde\tau_1,\widetilde\tau_1+\widetilde \tau_2[$, $\widetilde{\bm Y}^*$ is obtained by conditionally independence from $\widetilde{\bm Y}^*$ on $[\widetilde\tau_1,\widetilde\tau_1+\widetilde\tau_2[$, with initial value $\widetilde{\bm y}_0=\widetilde{\bm Y}^* (\widetilde\tau_1)$.  From step $m=1$, $\widetilde{\bm X}^*(\widetilde\tau_{1;n})\to \widetilde{\bm Y}^*(\widetilde\tau_1)$, and since $\widetilde{\bm Y}^*(\widetilde\tau_1)<\bm B$ and $\widetilde{\bm Y}^* \in H^k$, we can apply Lemma \ref{window1FPT}. Then, \eqref{eqboh} follows noting that \eqref{triplets} also holds if we reset the firing components $\widetilde\eta_{2;n}$ and $\widetilde\eta_{2}$, as done in \eqref{firingc}.
\item[Step $m>2$] It follows mimicking Step 2.
\end{enumerate}
\end{proof}

\section{Application to neural network modeling}\label{SectionX}
Membrane potential dynamics of neurons are determined by the arrival of excitatory and inhibitory postsynaptic potentials (PSPs) inputs that increase or decrease the membrane voltage. Different models account for different levels of complexity in the description of membrane potential dynamics. In LIF models, the membrane potential of a single neuron evolves according to a stochastic differential equation, with a drift term modeling the neuronal (deterministic) dynamics, e.g. input signals, spontaneous membrane decay, and the noise term accounting for  random dynamics of  incoming inputs. 

The first LIF model was proposed by Stein \cite{Ste65} to model the firing activity of single neurons which receive a very large number of inputs from separated sources, e.g. Purkinjie cells. The membrane potential evolution is given by \eqref{Stein} with $k=1$ when $X_n(t)$ is less than a firing threshold $B>x_{0;n}$, considered constant for simplicity. Each event of the excitatory process $N_n^+(t)$ depolarizes the membrane potential by $a_n>0$ and analogously the inhibition process $N_n^-(t)$ produces a hyperpolarization of size $b_n<0$. The values $a_n$ and $b_n$ represent the values of excitatory and inhibitory PSPs, respectively. Between events of input processes $N^+_n$ and $N^-_n$, $X_n$ decays exponentially to its resting potential $x_{0;n}$ with time constant $\theta$.  The firing mechanism was modeled as follows: a neuron releases a spike when its membrane potential attains the threshold value. Then the membrane potential is instantaneously reset to its starting value and the dynamics restarts. The intertime between two consecutive spikes, called interspike intervals (ISIs), are modeled as FPTs of the process through the boundary. Since the ISIs of the single neuron are independent and identically distributed, the underlying process is renewal. \\
In the following subsections we extend the one-dimensional  Stein model to the multivariate case to describe a neural network. We interpret all previous processes and theorems in the framework of neuroscience.

\subsection{Multivariate Stein model}
When $k>1$, \eqref{Stein} represents a multivariate generalization of the Stein model for the description of the sub-threshold membrane potential evolution of a network of $k$ neurons like Purkinjie cells. The synaptic inputs impinging on neuron $j$ are modeled by $N_{j;n}$, while $M_{A;n}$ models the synaptic inputs impinging on a cluster of neurons belonging to a set $A$. The presence of $M_{A;n}$ allows for simultaneous jumps for the corresponding set of neurons $A$ and determines a dependence between their membrane potential evolutions. We call this kind of structure \emph{cluster dynamics}  and we limit our paper to this type of dependence between neurons. Note that \eqref{Stein} might be rewritten in a more compact way, summing the Poisson processes with the same jump amplitudes. However, we prefer to distinguish between $N$ and $M$, to highlight their different role in determining the dependence structure. To  simplify the notation, we assume $\theta$ to be the same in all neurons. This is a common hypothesis since the resistance properties of the neuronal membrane are similar for different neurons \citep{TuckwellBook}. 
As for the univariate Stein, this proposed multivariate LIF model catches some physiological features of the neurons, namely the spontaneous decay of the membrane potential in absence of inputs and the effect of PSPs on the membrane potentials.
\
\subsection{Multivariate OU to model sub-threshold dynamics of neural network}
To make the multivariate Stein model mathematically tractable, we perform a diffusion limit. Theorem \ref{theo} guarantees that a multivariate OU process \eqref{OU} can be used to
 approximate a multivariate Stein when the frequency of PSPs increases and the contribution of the single postsynaptic potential becomes negligible with respect to the total input, i.e. for neural networks characterized by a large number of synapses. Being the diffusion limit of the multivariate Stein model, the OU  inherits both its biological meaning and  dependence structure. Indeed they have the same membrane time constant $\theta$, which is responsible for the exponential decay of the membrane potential. Moreover,  the terms $\mu_\cdot$ and $\sigma_\cdot$ of the OU are given by  \eqref{cond1} and \eqref{cond2} respectively, and thus they incorporate both frequencies and amplitudes of the jumps of the Poisson processes underlying the multivariate Stein model. 
Finally, if some neurons $j$ and $l$ belong to the same cluster $A$, their dynamics are related. This dependence is caught by the term $\sigma _{A}^{2}$ in the component $\psi_{jl}$ of the covariance matrix $\bm\psi$, which is not diagonal. This highlights the importance of having correlated noise in the model, and it represents a novelty in the framework of neural network models. Indeed, the dependence is commonly introduced in the drift term,  motivated by direct interactions between neurons, while the noise components are independent, see e.g. \cite{BBB,Brunel}. Here we ignore this last type of dependence to focus on cluster dynamics, but the proposed model can be further generalized introducing direct interactions between the $i$th and $j$th components, as noted in Remark \ref{Remark2.4}. 

\subsection{Firing neural network model and convergence of the spike trains}
In Section \ref{Section3}  we introduce the necessary mathematical tools to extend the single neuron firing mechanism to a network of $k$ neurons. Consider the sub-threshold membrane potential dynamics of a neural network described by a multivariate Stein model $\bm X_n$. A neuron $j$, $1\leq j\leq k$ releases a spike when the membrane potential attains its boundary level $B_j$. Whenever it fires, its membrane potential is instantaneously reset to its resting potential $r_{0j}<B_j$ and then its dynamics restart. Meanwhile, the other components are not reset but continue their evolutions. Since the inputs are modeled by stationary Poisson processes, the ISIs within each spike train are independent and identically distributed. Thus the single neuron firing mechanism holds for each component, which is described as a one-dimensional renewal Stein model. The firing neural network model is described by a multivariate process  behaving as the multivariate Stein process $\bm X_n$ in each time window between two consecutive passage times.  For this reason, we call this model, \emph{ multivariate firing Stein model} and we denote it $\bm X^*_n$. The ISIs of the components of the multivariate processes are neither independent nor identically distributed. 
We identify the spike epochs of the $j$th component of the Stein process, as the FPT of $X_{j,n}$ through the boundary $B_j$. The set of spike trains of all neurons corresponds to a multivariate point process with events given by the spikes. An alternative way of considering the simultaneously recorded spike trains is to overlap them and mark each spike with the component which generates it. Thus, we obtain the univariate point process $\bm\tau_n$ with marked events $\bm\eta_n$. The objects $\bm Y^*, \bm \tau $ and $\bm \eta$ are similarly defined for the multivariate OU process $\bm Y$, and we call $\bm Y^*$  \emph{multivariate firing OU process}.  Hence the models $\bm X^*_n$ and $\bm Y^*$ describe the membrane potential dynamics of a network of neurons with reset mechanism after a spike and thus are multivariate LIF models.

Finally, Theorem \ref{conspike} implies the convergence of the multivariate firing processes  $\bm X^*_n$ to $\bm Y^*$ and the convergence of the collection of marked spike train $(\bm\tau_n,\bm\eta_n)$ to $(\bm \tau,\bm\eta)$.  This guarantees that the neural code encoded in the FPTs is not lost in the diffusion limit. 

\subsection{Discussion}
As application of our mathematical findings, we developed a LIF model able to catch dependence features between spike trains in a neural networks characterized by large number of inputs from surrounding sources. \\
To make the model mathematically tractable, we introduced three assumptions: each neuron is identified with a point; Poisson inputs in \eqref{Stein} are independent; a firing neuron is instantaneously reset to its starting value. The first assumption characterizes univariate LIF models and has been recently assumed for two-compartmental neuronal model \cite{BendettoSacerdote}. We are aware that the Hodgkin-Huxley (HH) model and its variants are more biologically realistic than LIF. Indeed the HH model is a deterministic, macroscopic model describing the coupled evolution of the neural membrane potential and the averaged gating dynamics of Sodium and Potassium ion channels through a system of non-linear ordinary differential equations \cite{HH}. However, a mathematical relationship between Morris-Lecar model, i.e. a simplified version of HH model, and LIF models has been recently shown \cite{DitlevsenGreenwood}. This gives a (further) biological support to the use of LIF models and allows to avoid mathematical difficulties and  computationally expensive numerical implementations which are required for HH models.\\
The second assumption grounds on the description of the activity of each synapsis through a point process and it is also common to HH models, for which ion channels are modeled by independent Markov jump processes \cite{VB1991}. Physiological observations suggest that the behavior of each synapsis is weakly correlated with that of the others. Thanks to Palm-Kintchine Theorem, the overall neuron's input is described by two Poisson processes, one for the global inhibition and the other for the global excitation \cite{CapocelliRicciardi73}. \\
The third assumption has been introduced to simplify the notation, but it is not restrictive. Remark \ref{Remark4.1} guarantees
the convergence of the firing process and of the spike times in presence of delayed resets. Thus a refractory period can be introduced after each spike, increasing the biological realism of the model. Indeed after a spike, there is a time interval, called absolute refractory period, during which the spiking neuron cannot fire (while the others can), even in presence of strong stimulation \cite{TuckwellBook}.  

Having a multivariate LIF model for neural networks, several researches will be possible. First, one can simulate dependent spike trains from neural networks with known dependence structures. This allows to compare and test the reliability of  different existing statistical techniques for the detection of dependence structure between neurons, see e.g. \cite{STZ,gruen,Roman}.  Moreover, inspired by the techniques for the FPT problem of  univariate LIF models, one can develop analytical, numerical and statistical methods for the multivariate OU (or other diffusion processes) and its FPT problem, see e.g. \cite{STZ,STZsiam}. Furthermore, more biologically realistic LIF models for neural networks can be considered. Indeed Theorem \ref{conspike} can be applied to more general models such as Stein processes with direct interactions between neurons, Stein with reversal potential \cite{LanskyLanska} or birth and death processes with reversal potential \cite{GLNR}. 

Finally,  the application of our results  in the neuroscience framework is not limited to the case of LIF models. Thanks to Remark \ref{remark2}, Theorem \ref{conspike} can be applied to processes obtained through diffusion and fluid limits, i.e. both LIF and HH models. Since the HH model can be obtained as a  fluid limit \cite{PTW}, once a proper reset and firing mechanism is introduced, the convergence of the FPTs follows straightforwardly from our results. 

\section*{Acknowledgements} 
This paper is the natural extension of the researches started and encouraged by Professor Luigi Ricciardi on stochastic Leaky Integrate-and-Fire models for the description of firing activity of single neurons. This work is dedicated to his memory. \\
\\
The authors are grateful to Priscilla Greenwood for useful suggestions.
\\
\\
L.S. was supported by University of Torino (local grant 2012) and by project AMALFI -
Advanced Methodologies for the AnaLysis and management of the
Future Internet (Universit\`{a} di Torino/Compagnia di San
Paolo). The work is part of the Dynamical Systems Interdisciplinary Network, University of Copenhagen. 

\appendix
\section{Proofs of Section 2}\label{AppA}
To prove Lemma \ref{lemma1}, we first need to provide the characteristic triplet of $\bm{Z}_n$, as suggested in \cite{Jacod}.  The characteristic function of $\bm{Z}_n(t)$, is:
\begin{equation}\label{char}
\phi_{\bm{Z}_n(t)}(\bm{u})=\mathbb{E}\left[i\exp\left\{\sum_{j=1}^k u_j Z_{j;n}(t)\right\}\right],
\end{equation}
where $\bm{u}=(u_1,\ldots,u_k)\in \mathbb{R}^k$. We can write:
\begin{eqnarray}
\nonumber \sum_{j=1}^{k}u_{j}Z_{j;n}(t) &=&\sum_{j=1}^{k}u_{j}\left[ -\Gamma
_{j;n}t+\left( a_{n}N_{j;n}^{+}(t)+b_{n}N_{j;n}^{-}(t)\right) %
\right]   \\
&+&\sum_{A\in \mathcal{A}}G_{A}\left(
a_{n}M_{A;n}^{+}(t)+b_{n}M_{A;n}^{-}(t)\right),   \label{UZ} 
\end{eqnarray}
where $G_{A}=\sum_{j\in A}u_{j}$. Plugging \eqref{UZ} in \eqref{char} and since the processes in \eqref{UZ} are independent and Poisson distributed for each $n$, we get the characteristic function
\[
\phi_{\bm{Z}_n(t)}(\bm{u})=\exp\{t\rho_n(\bm{u})\},
\]
where%
\begin{eqnarray*}
\rho _{n}\left( \bm{u}\right)  &=&-i\sum_{j=1}^{k}u_{j}\Gamma
_{j;n}+\sum_{j=1}^{k}\alpha _{j;n}\left( e^{iu_{j}a_{n}}-1\right)
+\sum_{j=1}^{k}\beta _{j;n}\left( e^{iu_{j}b_{n}}-1\right)  \\
&&+\sum_{A\in \mathcal{A}}\lambda _{A;n}\left( e^{iG_{A}a_{n}}-1\right)
+\sum_{A\in \mathcal{A}}\omega _{A;n}\left( e^{iG_{A}b_{n}}-1\right) .
\end{eqnarray*}%
In \cite{Jacod}, convergence results are proved for $\rho_n(\bm{u})$ given by %
\[
\rho _{n}\left( \bm{u}\right) =i\bm{u\cdot b}_{n}-\frac{1}{2}\bm{%
u\cdot c}_{n}\bm{\cdot u+}\int_{\mathbb{R}^{k}\backslash 0}\left( e^{i%
\bm{u\cdot x}}-1-i\bm{u\cdot h}\left( \bm{x}\right) \right)
\,\nu _{n}\left( d\bm{x}\right) ,
\]%
(see Corollary II.4.19 in \cite{Jacod}), where $\bm{u\cdot v}=\sum_{j=1}^{k}u_{j}v_{j}$ and  $\bm{u\cdot
d\cdot v}=\sum_{j,l=1}^{k}u_{j}d_{jl}v_{l}$. The vector $\bm b_n$, the matrix $\bm c_n$ and the L\'{e}vy measure $\nu _{n}$ are known as characteristic triplet of the process. 
Here $\bm{h}:\mathbb{R}^{k}\rightarrow 
\mathbb{R}^{k}$ is an  arbitrary truncation function that is the same for all $n$, is bounded with compact support and satisfies $\bm{h}\left( \bm{x}\right) =%
\bm{x}$ in a neighborhood of $\bm{0}$. In our case, the triplet is 
\begin{enumerate}
\item $\nu _{n}$: finite measure concentrated on finitely many points,%
\begin{eqnarray*}
\nu _{n}\left( \left\{ \bm{x}:x_{j}=a_{n}\right\} \right) 
&=&\alpha _{j;n},\quad \left( 1\leq j\leq k,\neq 0\right);  \\
\nu _{n}\left( \left\{ \bm{x}:x_{j}=b_{n}\right\} \right) 
&=&\beta _{j;n},\quad \left( 1\leq j\leq k,\neq 0\right) ; \\
\nu _{n}\left( \left\{ \bm{x}:x_{j}=a_{n}\right\} \text{ for }j\in A\right)  &=&\lambda _{A;n},\quad \left(
A\in \mathcal{A},\neq 0\right) ; \\
\nu _{n}\left( \left\{ \bm{x}:x_{j}=b_{n}\right\} \text{ for }j\in A\right)  &=&\omega _{A;n},\quad \left( A\in 
\mathcal{A},\neq 0\right) .
\end{eqnarray*}%
All the non-specified $x_{j}$ are set to $0$, i.e. $
\left\{ \bm{x}:x_{j}=a_{n}\right\} =\left\{ \bm{x}
:x_{j}=a_{n}, \right.$ $\left. x_{l}=0\text{ for }l\neq j\right\} $. 
Since $a_{n}\rightarrow 0$ and $b_{n}\rightarrow 0$ when  $n$ is
sufficiently large, $\nu _{n}$ is concentrated on a finite subset of the
neighborhood of $\bm{0}$, where $\bm{h}\left( \bm{x}\right) =%
\bm{x}$. Without loss of generality, we may therefore, and shall, assume that $\bm{h}\left( \bm{x}\right) =\bm{x}$.
\item $\bm{c}_{n}=\bm{0}$.
\item $\bm{b}_{n}=-\bm{\Gamma }_{n}+\int \bm{h}\left( \bm{x}%
\right) \,\nu _{n}\left( d\bm{x}\right) $=0. Indeed, using $\bm{h}%
\left( \bm{x}\right) =\bm{x}$, we have
\begin{equation*}
b_{j;n} =-\Gamma _{j;n}+\left( \alpha _{j;n}a_{n}+\beta
_{j;n}b_{n}\right)+\sum_{A\in \mathcal{A}}\mathbbm{1}_{\{j\in A\}}\left( \lambda _{A;n}a_{n}+\omega _{A;n}b_{n}\right)=0.
\end{equation*}
\end{enumerate}
Having provided the triplet $(\bm{b}_n,\bm{c}_n,\nu_n)$, we can prove Lemma \ref{lemma1} as follows
\begin{proof}[Proof of Lemma \ref{lemma1}]
Use Theorem VII.3.4 in \cite{Jacod}. In our case, the weak convergence of $\bm{Z}_{n}$ to $\bm{W}$ follows if 
\begin{itemize}
\item [i.] $\bm{b}_{n}\rightarrow 
\bm{0}$; 
\item [ii.] $\widetilde{c}_{jl;n}:=\int x_{j}x_{l}\,\nu _{n}\left( d\bm{x}\right)  
\rightarrow \psi _{jl}$ for $1\leq j,l\leq k$; 
\item [iii.]$\int g\,d\nu _{n}\rightarrow 0$ for all $g\in
C_{1}\left( \mathbb{R}^{k}\right) $;
\item [iv.] $\bm{B}^n_t=t\bm b_n$ and $\bm{\tilde{C}}^n_t = t\widetilde{\bm c}_n$ converge uniformly to $\bm{B}_t$ and $\bm{\tilde{C}}_t$ respectively, on any compact interval $[0,t]$.
\end{itemize}
Here $C_{1}\left( \mathbb{R}^{k}\right) $ is defined in VII.2.7 in \cite{Jacod}. Since $\bm B_{t}^{n}=t\bm{b}_{n}$, the uniform convergence is evident. Furthermore,  $\widetilde{\bm C}_{t}^{n}=t\widetilde{\bm{c}}_{n}$ converges uniformly provided that condition [ii] holds. To prove [ii], we rewrite $\widetilde{c}_{jl;n}$ as follows
\begin{eqnarray}
\nonumber\widetilde{c}_{jl;n} &=&\sum_{i=1}^{k}\left( \mathbbm{1}_{\{i=l=j\}} \alpha _{j;n}a_{n}^{2}+\mathbbm{1}_{\{i=l=j\}}\beta
_{j;n}b_{n}^{2}\right)   +\sum_{A\in \mathcal{A}}\mathbbm{1}_{\{j,l\in A\}}\left( \lambda _{A;n}a_{n}^{2}+\omega _{A;n}b_{n}^{2}\right) 
\nonumber\\
\label{CONV}&=&\mathbbm{1}_{\{j=l\}}\sigma_{j;n}^2+\sum_{A\in \mathcal{A}}\mathbbm{1}_{\{j,l \in A\}}\sigma^2_{A;n}.
\end{eqnarray}%
Then, $\widetilde{c}_{jl;n} \to \psi_{jl}$ follows from the convergence assumptions $(\ref{cond3}), (\ref{cond4}), (\ref{cond1}), (\ref{cond2})$. \newline
Using Theorem VII.2.8 in \cite{Jacod}, we may show [iv] considering $g\in C_{3}\left( \mathbb{R}^{k}\right) $, i.e.
the space of bounded and continuous function $g:\mathbb{R}^{k}\rightarrow \mathbb{R}$
such that $g(\bm{x})=o\left( \left\vert \bm{x}\right\vert
^{2}\right) $ as $\bm{x}\rightarrow 0$. Here, $\left\vert \bm{x}%
\right\vert $ is the Euclidean norm. For $g\in C_{3}\left( \mathbb{R}^{k}\right) $
and $\varepsilon >0$, we have $\left\vert g(\bm{x})\right\vert \leq
\varepsilon \left\vert \bm{x}\right\vert ^{2}$ for $\left\vert \bm{x}\right\vert $ sufficiently small. Then
\[
\left\vert \int g\,d\nu _{n}\right\vert \leq \varepsilon \int \left\vert 
\bm{x}\right\vert ^{2}\,d\nu _{n}\rightarrow \varepsilon
\sum_{i=1}^{k}\psi _{ii}
\]%
by (\ref{CONV}), and $\int g\,d\nu _{n}\rightarrow 0$ follows. Indeed, since $\bm{W}$ is continuous, the L\'{e}vy measure $\nu $ for $\bm{W}$ is the null measure. 
\end{proof}
\begin{proof}[Proof of Theorem 1]
The $j$th component of $\bm{X}_{n}$ can be rewritten in terms of the $j$th component of $\bm{Z}_n$ as 
\begin{equation}\label{eq11}
X_{j; n}(t)=x_{0j;n}+ \int_0^t \left[-\frac{X_{j; n}(s)}{\theta}+\Gamma_{j; n} \right]ds+Z_{j; n}(t), \qquad 1\leq j\leq k.
\end{equation}
Solving it, we get
\begin{equation*} 
X_{j; n}(t)=x_{0j;n}e^{-\frac{t}{\theta}}+Z_{j; n}(t)-\frac{1}{\theta}\int_0^t e^{-(t-s)/\theta}Z_{j; n}(s)ds, \qquad 1\leq j\leq k.
\end{equation*}
Hence, $\bm{X}_n$ is a continuous functional of both $\bm x_{0;n}$ and $\bm{Z}_n$. Therefore, due to the continuous mapping theorem, the weak convergence of $\bm x_{0;n}$ (for hypothesis) and $\bm{Z}_n$ (from Lemma \ref{lemma1}) implies the weak convergence of $\bm{X}_n$. Moreover, \eqref{eq11} guarantees that the limit process of $\bm{X}_n$ is that defined by \eqref{OU}. 
\end{proof}

\end{document}